\documentclass[11pt]{article}

\usepackage[utf8]{inputenc}
\usepackage[english]{babel}
\usepackage{amsmath,amssymb,amsthm,amsfonts,amsbsy,bbm}
\usepackage{hyperref,fullpage}
\usepackage[capitalise]{cleveref}
\usepackage{subcaption}
\usepackage{multicol}
\usepackage{comment}

\usepackage{enumerate}

\usepackage{tikz}
\usetikzlibrary{patterns,decorations}
\usetikzlibrary{calc}
\DeclareMathSymbol{\shortminus}{\mathbin}{AMSa}{"39}

\allowdisplaybreaks
\usepackage[format=plain,
            labelfont={bf,it},
            textfont=it]{caption}
\crefname{equation}{}{}

\newtheorem{theorem}{Theorem}[section]
\newtheorem{lemma}[theorem]{Lemma}
\newtheorem{definition}[theorem]{Definition}

\newtheorem{corollary}[theorem]{Corollary}

\newtheorem{conjecture}[theorem]{Conjecture}

\newtheorem{proposition}[theorem]{Proposition}

\newcommand{\N}{\mathcal{N}}

\newcommand{\mG}{\mathcal{G}}

\newcommand{\p}{\mathcal{P}}

\newcommand{\abs}[1]{\left\lvert{#1}\right\rvert}
\newcommand{\floor}[1]{\left\lfloor{#1}\right\rfloor}

\DeclareMathOperator{\ex}{ex}

\newcommand{\ao}{%
  \:\mathrel{%
    \vcenter{\offinterlineskip
      \ialign{##\cr$\alpha$\cr\noalign{\kern1.7pt}$\omega$\cr} %
    }%
  }%
  \!
}

\title{Generalized Planar Tur\'an Numbers}

 \author{
Ervin Gy\H{o}ri\thanks{Alfr\'ed R\'enyi Institute of Mathematics, Hungarian Academy of Sciences. email: \texttt{gyori.ervin@renyi.mta.hu}} \footnotemark[3] \and
Addisu Paulos\thanks{Addis Ababa University, Addis Ababa. email: \texttt{addisu\_2004@yahoo.com}} \thanks{Central European University, Budapest.} \and
Nika Salia\thanks{Alfr\'ed R\'enyi Institute of Mathematics, Hungarian Academy of Sciences. email: \texttt{salia.nika@renyi.hu}} \footnotemark[3] \and 
Casey Tompkins\thanks{Discrete Mathematics Group, Institute for Basic Science (IBS), Daejeon, Republic of Korea.} \thanks{Karlsruhe Institute of Technology, Karlsruhe, Germany. email:\texttt{ctompkins496@gmail.com}} \and 
Oscar Zamora \footnotemark[3] \thanks{Universidad de Costa Rica, San Jos\'e. email: \texttt{oscarz93@yahoo.es}}
\and
}

\begin{document}
\maketitle

\begin{abstract}
In a generalized Tur\'an problem, we are given graphs $H$ and $F$ and seek to maximize the number of copies of $H$ in an $F$-free graph of order $n$. We consider generalized Tur\'an problems where the host graph is planar. In particular we obtain the order of magnitude of the maximum number of copies of a fixed tree in a planar graph containing no even cycle of length at most $2\ell$, for all $\ell$, $\ell \geq 1$. We obtain the order of magnitude of the maximum number of cycles of a given length in a planar $C_4$-free graph. 
An exact result is given for the maximum number of $5$-cycles in a $C_4$-free planar graph. 
Multiple conjectures are also introduced. 
\end{abstract}

\section{Introduction}
For a fixed graph $F$, the classical extremal function $\ex(n,F)$ is defined to be the maximum number of edges possible in an $n$-vertex graph not containing $F$ as a subgraph.  This function naturally generalizes to a setting where, rather than edges,  we maximize the number of copies of a given graph $H$ in an $n$-vertex $F$-free graph.  Following Alon and Shikhelman~\cite{ALS2016} (see also~\cite{add}), we denote this more general function by $\ex(n,H,F)$.
Problems of this type have a long history beginning with a result of Zykov~\cite{zykov} (and later independently Erd\H{o}s~\cite{erdos}) who determined the value of $\ex(n,K_r,K_t)$ for any pair of cliques. After these initial results a variety of other results of this type were obtained, perhaps the most well-known of which being the determination of $\ex(n,C_5,C_3)$ by Hatami, Hladk\'y, Kr\'al, Norine and Razborov~\cite{HHKNR2013} and independently by Grzesik~\cite{G2012}. Many other results about generalized extremal numbers have also been obtained. 
For example, see~\cite{bb,gerbpalm,let,luo}.


In a different direction, extremal problems have been considered for host graphs other than $K_n$. Examples include the Zarankiewicz problem where the host graph is taken to be a complete bipartite graph, or extremal problems on the hypercube $Q_n$ initiated by Erd\H{o}s~\cite{cube}. 
More recently, extremal problems have been considered where the host graph is taken to be a planar graph. 
For a given graph $F$, let us denote the maximum number of edges in an $n$-vertex $F$-free planar graph by $\ex_{\p}(n,F)$. 
This topic was initiated by Dowden in ~\cite{dowden} who determined  $\ex_{\p}(n,C_4)$ and $\ex_{\p}(n,C_5)$.  
A variety of other forbidden graphs $F$ including stars, wheels and fans were considered by Lan, Shi and Song~\cite{lan2}. The case of theta graphs was considered in Lan, Shi and Song~\cite{lan3}, and the case of short paths was considered by Lan and Shi in~\cite{lan1}. Some closely related anti-Ramsey problems were considered in~\cite{lan4} and~\cite{lan5}.  

Another direction of research which has been considered is maximizing the number of copies of a given graph in an $n$-vertex planar graph.  
Hakimi and Schmeichel~\cite{hakimi} determined the maximum number of triangle and $C_4$ copies possible in a planar graph.  
These results were extended by Alon and Caro~\cite{alon} who determined the maximum number of $K_{1,t}$ and $K_{2,t}$ copies possible. 
Alon and Caro also proved that the maximum number of $K_4$ copies in a planar graph is $n-3$. Resolving a conjecture attributed to Perles in~\cite{alon},  Wormald~\cite{Wormald} proved that every 3-connected graph $H$ occurs at most $c_H n$ times in an $n$-vertex planar graph for some constant $c_H$ depending on $H$ (this result was proved again in a different way by Eppstein~\cite{eppstein}).  
A simple argument shows that graphs with at least 3 vertices which are at most 2-connected will occur at least quadraticly many times in a planar graph. 
Thus the preceding result of Wormald and Eppstein provides a characterization of graphs which can occur at most $O(n)$ times in a planar graph.  Recently, resolving a conjecture of Hakimi and Schmeichel~\cite{hakimi}, the maximum number of $5$-cycles in a planar graph was determined in~\cite{c5}. 
Similarly, the maximum number of paths of length three was determined in~\cite{p3}. 

It is interesting to note that the problem of maximizing $H$ copies in a planar graph is in some sense a special case of the problem of Alon and Shikelman~\cite{ALS2016}.  Indeed, for a given graph $H$, and the collection $\mathcal{F}$ of subdivisions of $K_5$ and $K_{3,3}$, it follows from Kuratowski's~\cite{Kuratowski} theorem that $\ex(n,H,\mathcal{F})$ is equal to the maximum number of $H$-copies in an $n$-vertex planar graph.

In this paper we will consider a common generalization of the types of problems mentioned above. To this end, let $\ex_{\p}(n,H,\mathcal{F})$ denote the maximum number of copies of $H$ possible in an $n$-vertex planar graph containing no graph $F \in \mathcal{F}$ as a subgraph (we write simply $\ex_{\p}(n,H,F)$ in the case $\mathcal{F} = \{F\}$).  In particular, the problems considered in the preceding paragraphs correspond to the special cases of $\ex_{\p}(n,K_2,F)$ and $\ex_{\p}(n,H,\emptyset)$, for given graphs $F$ and $H$.

\section{Notation and Results}
For a graph $G$, we denote the vertex set and edge set of $G$ by $V(G)$ and $E(G)$ respectively. Also we let $v(G)=|V(G)|$ and $e(G)=|E(G)|$. For a vertex $v\in V(G)$, the degree of $v$ is denoted by $d(v)$, and the minimum degree of the graph $G$ is denoted by $\delta(G)$. We use the notation $C_t$ and $K_t$ respectively for the cycle and complete graph of order~$t$.  We denote the path of length $t$ (that is, the path with $t$ edges) by $P_t$.  For a graph $G$ and an independent set $S \subseteq V(G)$, the graph obtained by blowing $S$ up by $t$ is the graph formed by replacing each vertex $v\in S$ with $t$ vertices each with the same neighbors as $v$.

In the general case, Eppstein~\cite{eppstein} asked whether for all $H$, the maximum number of copies of $H$ possible in a minor-closed family of $n$-vertex graphs is an integer power of $n$. We state a restricted version of his problem as a conjecture in the planar case, and later generalize it further.  
\begin{conjecture} \label{Con:integer_empty}
For every graph $H$, there exists a non-negative integer $k$, such that \[\ex_{\p}(n,H,\emptyset)=\Theta(n^k).\]
\end{conjecture}

We verify this conjecture in the case of trees. To state our result we require some notation.  

\begin{definition}\label{Definition_Beta_i}
For a graph  $H$ and an integer $i$, $i\ge 1$, let $\beta_i(H)$ be the maximum number of components in an  induced subgraph of $H$ containing only two types of components:~(1)~isolated vertices which have degree one  in $H$ and (2)~paths of length $i-1$ consisting only of vertices of degree two from $H$.   
\end{definition}
In particular in the case when $i=1$, we are interested in the maximum size of an independent set in the graph consisting of vertices of degree at most 2.  For simplicity, we let $\beta(G) := \beta_1(G)$ for any graph $G$.  Then we have the following result for trees.

\begin{theorem}\label{treeorder}
Let $T$ be a tree, then
\[
\ex_{\mathcal{P}}(n,T,\emptyset)=\Theta(n^{\beta(T)}).
\]
\end{theorem}

For any graph $H$, let $\alpha(H)$ be the independence number of $H$. In the general case we have the following upper bound. 

\begin{theorem}
\label{planargen}
Let $H$ be any graph, then
\[
\ex_{\mathcal{P}}(n,H,\emptyset)=O(n^{\alpha(H)}).
\]
\end{theorem}
This theorem will follow as an immediate consequence of results we prove about degenerate graphs in Section~\ref{degenerate}. As a corollary of Theorem~\ref{planargen} we obtain the order of magnitude of the maximum number of cycles (note that this result was already obtained by Hakimi and Schmeichel in~\cite{hakimi}).

\begin{corollary}\label{ckbound}
For all $k \ge 3$, we have
\[
\ex_{\p}(n,C_k,\emptyset) = \Theta(n^{\floor{k/2}}).
\]
\end{corollary}
The lower bound is attained by taking a cycle $C_k$ and blowing up a maximum sized independent set by $\floor{2n/k}-1$.  Note that the constant in the asymptotic notation may depend on $k$, and this construction contains asymptotically $\left(\frac{2n}{k}\right)^{\floor{k/2}}$ copies of $C_k$.

Next we consider the case when the set of forbidden graphs is nonempty.  In this case, we pose a conjecture which generalizes Conjecture~\ref{Con:integer_empty}.  

\begin{conjecture}
For all finite sets of graphs $\mathcal{F}$ and for all graphs $H$, we have
\[
\ex_{\p}(n,H,\mathcal{F}) = \Theta(n^k),
\]
for some integer $k$.
\end{conjecture}

We consider a variation of Theorem~\ref{treeorder} for the case when $C_4,C_6,\dots,C_{2\ell}$ are forbidden. 
We prove the following.




\begin{theorem}
\label{eventree}
For any tree $T$, we have 
\[
\ex_{\p}(n,T,\{C_4,C_6,\dots,C_{2\ell}\})= \Theta(n^{\beta_\ell(T)}).
\]
\end{theorem}

The lower bound in Theorem~\ref{eventree} is given as follows.  Take an induced subgraph of $T$ consisting of $\beta_\ell(T)$ components as described in Definition~\ref{Definition_Beta_i}. Replace each path (including the ones of length 0) by $\Omega(n)$ paths of the same length with endpoints joined to the same neighbors as the corresponding paths in $T$ and number of vertices summing to $n$.  The resulting graph has $\Omega(n^{\beta_\ell(T)})$ copies of $T$ and contains no cycle in the set $\{C_4,C_6,\dots,C_{2\ell}\}$. In fact, we believe that a construction of this form should yield the correct asymptotic value of $\ex_{\p}(n,T,\{C_4,C_6,\dots,C_{2\ell}\})$, but our proof yields the order of magnitude.  We also have the following exact result for maximizing the number of $C_5$ copies in a $C_4$-free planar graph.

\begin{theorem}\label{thm:C_4_C_5}
 For all $n\geq 4$, $n\neq 6$, we have \[\ex_{\p}(n,C_5,\{C_4\})=n-4.\]  
\end{theorem}

Moreover, we determine the order of magnitude of $\ex_{\p}(n,C_k,\{C_4\})$ for every $k$.  We obtain the following result.

\begin{theorem}\label{c4cycle}
For all $k \ge 5$, we have
\[
\ex_{\p}(n,C_k,\{C_4\}) = \Theta(n^{\floor{k/3}}).
\]
\end{theorem}

We conjecture that in fact a much more general result holds.  

\begin{conjecture} \label{ckgenconjecture}
For sufficiently large $k$, we have
\[
\ex_{\p}(n,C_k,\{C_4,C_6,\dots,C_{2\ell}\}) = \Theta(n^{\floor{\frac{k}{\ell+1}}}).
\]
\end{conjecture}
A construction for a lower bound in Conjecture~\ref{ckgenconjecture} is similar to that of Theorem~\ref{eventree}.  Namely, we note that $\beta_{\ell}(C_k) = \floor{\frac{k}{\ell+1}}$, and replace each of the $\beta_{\ell}(C_k)$ paths with $\Omega(n)$ paths of the same length joined to the corresponding pair of vertices.  The proof of Theorem~\ref{c4cycle} can be adapted to resolve Conjecure~\ref{ckgenconjecture} in the cases when $k$ is congruent to 0, 1 or 2 modulo $\ell+1$.  

We conclude this section by contrasting our results in the planar case with the known results in the general case. It was shown in~\cite{gen1} and~\cite{gen2} that $\ex(n,C_k,C_4) = \Theta(n^{k/2})$.  This result is in stark contrast to our results in the planar case in two ways. First, in the planar case the order of magnitude is always an integer power of $n$, and second in the planar case we have $k/3$ rather than $k/2$ in the exponent.

       \section{General Upper Bounds for Degenerate Graph Classes}
       \label{degenerate}

     Alon and Shikhelman~\cite{ALS2016} proved that for any bipartite graph $H$ and tree $T$ we have {$\ex(n,H,T) = O(n^{\alpha(H)})$}, where $\alpha(H)$ is the independence number of $H$.  
     This result was extended to all graphs~$H$ in~\cite{plpk}. 
     Since the extremal number of a tree $T$ is linear in $n$, it follows that any $T$-free graph has a vertex of degree at most $c_T$, a constant depending on $T$.  
     Call a graph $G$  $c$-degenerate if every subgraph of $G$ contains a vertex of degree at most $c$.
     The proof from~\cite{plpk} can easily be extended to work for the class of  $c$-degenerate graphs.
     We now present a proof of this theorem for completeness.  
     
    First we introduce some notation.  For given graphs $G$ and $H$, let $\N(H,G)$ denote the number of copies of $H$ in $G$. Let $\mG_c$ denote the class of $c$-degenerate graphs, and let \[f_c(n,H):= \max\{\N(H,G): G \in \mG_c, v(G)=n\}.\]

\begin{proposition}
\label{clique}
 $f_c(n,K_r) = O(n)$, where the constant depends only on $r$ and $c$.
\end{proposition}

\begin{proof} We proceed by induction on $r$.  For $r=1$ the result is clear, so assume $r>1$ and that $f_c(n,K_{r-1}) \le C_{r-1}n$ for a constant $C_{r-1}$. Let $G$ be an $n$-vertex graph in $\mG_c,$ then  we have that \begin{displaymath}
r\N(K_r,G) = \sum_{v \in V(G)} \N(K_{r-1},G[N(v)]) \le \sum_{v \in V(G)} C_{r-1} d(v) = O(e(G)) = O(n). \qedhere \end{displaymath} 
\end{proof}

Theorem~\ref{planargen} follows as a simple consequence of the following lemma which will be proven by induction on $\alpha(H)$. 

\begin{lemma}
\label{dif}
For any graph $H$
\begin{displaymath}f_c(n+1,H) - f_c(n,H) = O(n^{\alpha(H)-1}).\end{displaymath} Here, the constant given by the $O$ notation depends only on $H$ and $c$.
\end{lemma}

We start by proving the following well-known fact.

\begin{proposition}
\label{simpleprop}
Let $H$ be a graph, and let $u$ be a vertex of $H$. If $H'$ is the graph obtained from $H$ by removing $u$ together with its neighborhood, then $\alpha(H') \leq \alpha(H)-1.$ 
\end{proposition}

\begin{proof}
If $X$ is a maximal independent set in $H'$, then since no neighbor of $u$ is in $X$, the set $X\cup\{u\}$ is independent in $H$ and so $\alpha(H') + 1 \leq \alpha(H).$
\end{proof}

We are now ready to prove Lemma \ref{dif}.

\begin{proof}[Proof of Lemma \ref{dif}]

 To estimate $f_c(n+1,H) - f_c(n,H)$, we will start with a graph $G \in \mG_c$ on $n+1$ vertices with the maximum number of copies of $H$.
 We know that $\delta(G)  \le c$. 
 Let $v$ be a vertex of minimum degree in $G$. We will estimate the number of copies of $H$ in $G$ containing $v$ as a vertex.

Let $V(H) = \{u_1,u_2,\dots,u_{v(H)}\}$, and let $H_i$ be the graph obtained by removing $u_i$  together with its neighbors. 
By Proposition \ref{simpleprop}, we know that $\alpha(H_i) \leq \alpha(H)-1$.  
Now for each copy of $H$ using $v$ as a vertex, $v$ must play the role of some $u_i$, and the neighbors of $u_i$ must be embedded in the neighborhood of $v$.  
It follows that the other vertices of $H$, that is the vertices of $H_i$, must be embedded in some way in the remaining vertices of $G$.
We have to choose $d_H(u_i)$ vertices in $N(v)$, so the number of copies of $H$ using $v$ is at most 
\begin{displaymath}
\sum_{i=1}^{v(H)} d(v)^{d_H(u_i)}\N(H_i,G) \leq \sum_{i=1}^{v(H)} c^{d_H(u_i)}\N(H_i,G) = \sum_{i=1}^{v(H)}O_{{H_i}}(n^{\alpha(H_i)}) = O(n^{\alpha(H)-1}).  \end{displaymath}

Thus, if $G'$ is the graph obtained from $G$ by removing $v$, we have that 
\[f_c(n+1,H) = \N(H,G) = \N(H,G') + O(n^{\alpha(H)-1}) \leq f_c(n,H) + O(n^{\alpha(H)-1}).\qedhere\] 
\end{proof}



\section{The Number of Trees in Planar Graphs} 
    In this section we prove Theorem~\ref{treeorder}.  
First we provide the lower bound, $\ex_{\mathcal{P}}(n,T,\emptyset)=\Omega(n^{\beta(T)})$. 
    Observe that, if a given graph is planar and one blows up a  set of independent vertices each of  degree at most two, then resulting graph is also planar.  
    Therefore the following construction provides the desired lower bound. Given a tree $T$, fix an independent set $S$ of size $\beta(T)$ which contains vertices of degree at most two (as in Definition~\ref{Definition_Beta_i}) and blow up this set by  $\floor{\frac{n}{2\beta(T)}}$.  
    The resulting graph is  planar with at most $n$ vertices, when $n$ is sufficiently large, and contains $\Omega(n^{\beta(T)})$ copies of the tree $T$.


Observe that  we have $\ex_{\mathcal{P}}(n,P_k,\emptyset)= O(n^{\alpha(P_k)})$ from Theorem \ref{planargen}, where $P_k$ denotes path of length $k$. Even more we have $\alpha(P_k)=\beta(P_k)$ from Definition \ref{Definition_Beta_i}. Therefore 
we have the following simple proposition.

\begin{proposition}\label{path}
    $\ex_{\mathcal{P}}(n,P_k,\emptyset)=\Theta(n^{\beta(P_k)})$.
\end{proposition}

Hence Theorem \ref{treeorder} holds for paths.
To show that Theorem \ref{treeorder} holds for any tree, we are going to use the following lemma.

           \begin{lemma}\label{Constant-Lemma} For a given planar graph $G$. Let $v$, $u$ and $w$ be fixed vertices in $G$ and let $n_1$, $n_2$ and $n_3$ be non negative integers.
           The number of vertices $x$, such that, there are three internally disjoint paths from $x$ to $v$, from $x$ to $u$ and from $x$ to $w$ of length $n_1$, $n_2$ and $n_3$, respectively, is bounded by a constant $C:=C(n_1,n_2,n_3).$
      \end{lemma}

       \begin{proof}
       Suppose we have a planar embedding of $G$.
       
       The proof will be by induction on $n_1+n_2$. The result is trivial if either $n_1$ or $n_2$ is equal to 0.
      So suppose that $n_1+n_2\geq 2$, and that the result holds for any pair with smaller sum.
 Consider a maximal set $\p$, of internally vertex disjoint paths $v, v^{i}_2,\dots, v^{i}_{n_1},a_i ,u^{i}_2,\dots,u^{i}_{n_2},u$, where each $a_i$ is such that there exist a length $n_3$ path from $a_i$ to $w$ which does not contain any of the vertices $v, v^{i}_2,\dots, v^{i}_{n_1},u^{i}_{n_2},\dots,u^{i}_{2},u$.
       Let us denote the set of $a_i$ in these paths by $A$.  
      Observe that the paths from $\p$ divide the plane into $\abs{A}$ regions $R_1,R_2,\dots,R_{\abs{A}}$.
     Since the vertex $w$ is fixed, it is in one of the regions, and there is  a path of length $n_3$ from $w$ to each  vertex of $A$, not using the vertices $v$ and $u$.  
     Thus $\abs{A} \leq  2n_3+1$. 
     
Now let $Y$ to be the set of $\abs{A}(n_1+n_2-1)+2$ vertices that appear in some path from $\p$, and
let $X$ be the set of those vertices $x$ in $G$ which are not in $Y$ such that there exist three internally disjoint paths from $x$ to $v$, from $x$ to $u$ and from $x$ to $w$ of length $n_1$, $n_2$ and $n_3$, respectively. 
It is sufficient to bound $\abs{X}+\abs{Y}$ by a constant depending on $n_1$, $n_2$ and $n_3$. If $X=\emptyset$ we immediately have the required bound.
Suppose $X$ is nonempty and let $x \in X$, and let $P_1 = v,v_2,\dots,v_{n_1},x$ and $P_2 = x,u_{n_2},\dots,u_2,u$ be two of the three internally disjoint paths from $x$. 
Let $v'$ and $u'$ be the first vertex (closest to $x$ in $P_i$) in the intersection of $Y$ with $P_1$ and $P_2$, respectively. Note that it is possible for $v'$ to be $v$ or $u'$ to be $u$, but by the definition of $Y$ and $\p$, is not possible for both to happen simultaneously. 
Then the vertex $x$ is such that there exist three internally disjoint paths from $x$ to $v'$, from $x$ to $u'$ and from $x$ to $w$ of length $n_1',n_2'$ and $n_3$ respectively, where $1 \leq n_i' \leq n_i$ for $i=1,2$ and $1 \leq n_3$, with the additional property that $n_1'+n_2' < n_1 + n_2$. 
Therefore, setting $\displaystyle C' = \max_{n_1'+n_2' < n_1+n_2 } C(n_1',n_2',n_3),$ we have that 
\[\displaystyle\abs{X} \leq \binom{\abs{A}(n_1+n_2-1)+2}{2}C' \leq \binom{(2n_3+1)(n_1+n_2-1)+2}{2}C' .\]  Thus, $\abs{X} + \abs{Y}$ is bounded and so the lemma holds.       \end{proof}

Note that Lemma~\ref{Constant-Lemma} implies in particular that if  $G$ is a planar graph and $T$ is a tree with $s\geq 3$ leaves $x_1,x_2,\dots,x_s$. 
Then for any vertex $x\in V(T)$ of degree at least 3 and  $v_1,v_2,\dots,v_s \in V(G)$, the number of vertices $v\in V(G)$, such that, there exists a copy of $T$ where $x$ is embed in $v$ and $x_i$ is embed in $v_i$, for $i=1,\dots,s$, is bounded by a constant that does not depend on $G$. That is since we are able to find three different leaves such that the paths from $x$ to each of these leaves are internally disjoint.

At this point we are ready to prove Theorem \ref{treeorder}.
\begin{proof}[Proof of Theorem~\ref{treeorder}]
We may assume $T$ is not a path otherwise we are done, from  Proposition~\ref{path}. 

Let $G$ be an $n$-vertex planar graph.
Let $A$ be the set of vertices of degree at least 3 in $T$, and let $T_1,T_2,\dots,T_k$ be the connected components of the graph induced by $V(T)\setminus A$ (the set~$A$ is non-empty since $T$ is not a path). 
Observe that since $A$ has every vertex of degree at least 3, then  every vertex of $T_i$ has degree at most 2 in both $T_i$ and $T$, so we have 
\[
\beta(T) =\beta\left(\bigcup_{i=1}^k T_i\right) =\sum_{i=1}^k \beta(T_i).
\]
Moreover the trees $T_i$ are paths and so $\N(T_i,G) = O(n^{\beta(T_i)})$, from Proposition~\ref{path}. 
Then for any embedding of the trees $T_i$ by Lemma\ref{Constant-Lemma}, there is a constant number of ways to complete the embedding of $A$ to a copy of $T$. Therefore the number of copies of $T$ is bounded by $O(n^{\sum_{i=1}^k\beta(T_i)})= O(n^{\beta(T)}).$
\end{proof}

    \section[]{The number of $C_5$'s in $C_4$-free Planar Graphs}
    
   In this section we are going to prove Theorem~\ref{thm:C_4_C_5}, namely that for all $n$, $n\geq 4$, $n\neq 6$, we have
   \[\ex_{\p}(n,C_5,\{C_4\})=n-4.\]
    \begin{proof}
    
    We begin by providing the lower bound.  Let $n=5+3t+2s$ for some nonnegative integers $s$ and $t$. (Note that when $n=6$, it is easy to verify that there can be at most one pentagon, thus $\ex_{\p}(6,C_5,\{C_4\})=1$.)  The construction is as follows: Take a pentagon $x_1,x_2,x_3,x_4,x_5$, as well as $t$ internally vertex disjoint paths $y_3^i,y_4^i,y_5^i$, $1 \le i \le t$ between $x_1$ and $x_2$ and add the edges so that $x_4,y_4^1,\dots,y_4^t$ forms a path.  Next take a disjoint path $z_1,z_2,\dots,z_{2s}$.  Add the edge from $z_1$ to $x_1$ and the edges from $z_i$ to $x_5$, for odd $i \equiv 0,1 \, (\bmod \,0 4)$, and $z_i$ to $x_3$, for $i  \equiv 2,3 \, (\bmod \, 4)$.  
    See Figure~\ref{combined} for an example of an extremal graph. 
    
    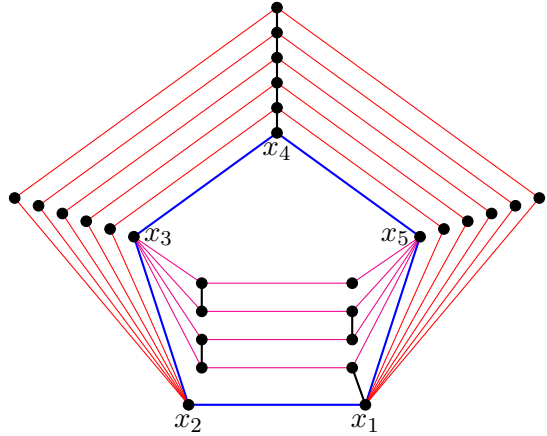
\begin{figure}[h]
    \centering
\begin{tikzpicture}
\pgfmathsetmacro{\a}{5}
\pgfmathsetmacro{\b}{4}
\pgfmathsetmacro{\c}{\b-1}
\draw[thick] (90:2) -- (90:{2+\a/3});
\draw[blue, thick]  (234:2) -- (306:2) -- (18:2)--(90:2) -- (162:2) -- (234:2);
\foreach \x in {1,2,...,\a}{
\ifnum\x > \a{
                            }  
    \else {
\draw[red] (306:2) -- (18:{2+\x/3})--(90:{2+\x/3}) -- (162:{2+\x/3}) -- (234:2);
\filldraw  (18:{2+\x/3}) circle (2pt) (90:{2+\x/3}) circle (2pt) (162:{2+\x/3}) circle (2pt); 
}
\fi
}

\draw[thick] (1,{1.5/\b-1.5}) -- (306:2);

\foreach \x in {1,2,...,\b}{
\ifnum\x > \b{
                            }  
    \else {

\draw[magenta] (162:2)--(-1,{1.5*\x/\b-1.5})  -- (1,{1.5*\x/\b-1.5})-- (18:2);
\filldraw (-1,{1.5*\x/\b-1.5}) circle (2pt) (1,{1.5*\x/\b-1.5}) circle (2pt);
}
\fi
}
\foreach \x in {1,3,...,\b}{
\ifnum\x > \c{
                            }  
    \else {
\draw[thick] (-1,{1.5*\x/\b-1.5})  -- (-1,{1.5*(\x+1)/\b-1.5});
}
\fi       
}
\foreach \x in {2,4,...,\b}{
\ifnum\x > \c{
                            }  
    \else {
\draw[thick] (1,{1.5*\x/\b-1.5})  -- (1,{1.5*(\x+1)/\b-1.5});
}
\fi       
}

\filldraw (234:2) node[below]{$x_2$} circle (2pt) (306:2) node[below]{$x_1$} circle (2pt) (18:2) node[left]{$x_5$} circle (2pt) (90:2) node[below]{$x_4$} circle (2pt) (162:2) node[right]{$x_3$}circle (2pt);

\end{tikzpicture}
    \caption{Example of an extremal graph for Theorem~\ref{thm:C_4_C_5}.}
    \label{combined}
\end{figure}

Now we are going to prove by induction that   $\ex_{\p}(n,C_5,\{C_4\}) \leq n-4$. The proof proceeds by induction on $n$ with the base cases \sloppy $\ex_{\p}(4,C_5,\{C_4\})=0$, $\ex_{\p}(5,C_5,\{C_4\})=1$ and $\ex_{\p}(6,C_5,\{C_4\})=1$. Let $G$ be an $n$ vertex, $C_4$ free planar graph with $n \ge 7$. Without loss of generality we may assume $G$ is connected.  Consider an embedding of $G$ on the plane. To prove an upper bound for Theorem~\ref{thm:C_4_C_5}, we take a planar embedding of $G$ and consider two cases.

\textbf{Case 1:} \emph{All pentagons in  $G$ are faces}.  
Remove  an edge from each triangular face. Observe that two triangular faces do not share an edge since $G$ is $C_4$-free. 
Let us denote the resulting graph by $G'$. 
The graph $G'$  is a planar graph with $n$ vertices, such that each face has at least 5 edges.  
Moreover, the total number of faces of $G'$ is at least the number of pentagonal faces in $G$. 
Thus if $f$ is the number of faces of the graph $G'$ and $e$ is the number of edges, then $5f\leq 2e$. Using  Euler's formula, $f+n=e+2$, we get $f\leq \frac{2}{3}n-\frac{4}{3}$. 
Thus we have that the number of $C_5$ copies in $G$ is at most $n-4$, since $n\geq 7$.

\textbf{Case 2:} \emph{There is a pentagon in  $G$ which is not a face}. 
Let $P$ be a non-facial pentagon in $G$. The pentagon $P$, cuts the plane into two regions. 
Let us denote the subgraph of $G$ in the inner and outer regions of $P$ by $G_{1}$ and $G_2$, respectively, where both graphs include the vertices of the pentagon. Assume that $G_1$ and $G_2$ have $n_1$ and $n_2$ vertices, respectively. 
Thus we have $n=n_1+n_2-5$, where $n_1$ and $n_2$ are both non-zero and less than $n$.

By the induction hypothesis, the number of pentagons in $G_{1}$ and $G_{2}$ is at most $n_1-4$ and $n_2-4$, respectively.  
Even more there are no pentagons crossing $P$ in $G$, since $G$ is $C_4$-free, hence it follows that every pentagon of $G$ is a pentagon of $G_{1}$ or  $G_{2}$. 
Since the pentagon $P$ is in both graphs, we get that the number of pentagons in $G$ is at most $n_1-4+n_2-4-1=n+5-9=n-4$, completing the proof. 
\end{proof}

     \section[]{The Order of Magnitude of the Maximum Number of Cycles of Length $k$ in a $C_4$-free Planar Graph} \label{secck}
       
We have seen in Corollary~\ref{ckbound} that $\ex_{\p}(n,C_k,\emptyset) = \Theta(n^{\floor{k/2}})$ follows immediately from Theorem~\ref{planargen}.  We now consider the case when $C_4$ is forbidden and prove Theorem~\ref{c4cycle} which states that $\ex_{\p}(n,C_k,\{C_4\}) = \Theta(n^{\floor{k/3}})$.
\begin{proof}
For the construction we take a cycle $C_k$ and find an induced matching of size $\floor{k/3}$.  
Next we replace each edge in this matching with $\frac{n-k}{2\floor{k/3}}$ edges each adjacent to the same pair of vertices as the original edge.  
This graph clearly has $\Theta(n^{\floor{\frac{k}{3}}})$ copies of $C_k$ and at most $n$ vertices, when $n$ is sufficiently large.  

We will now prove the upper bound.
It is well-known that every planar graph contains a vertex of degree at most $5$.  
It was proved in~\cite{mdegree} that a $C_4$-free planar graph with minimum degree at least 2 contains an edge $\{x,y\}$ such that $d(x)+d(y)$ is at most 8.  
This result was improved to 7 in~\cite{mdeg7}, which is best possible. 
Note that, to prove Theorem~\ref{c4cycle}, we may assume the graph has no vertices of degree one, since such vertices do not contribute to any $k$-cycles.  
We will distinguish cases based on the value of $k$ modulo~3. 
When $k$ is equal to $0$ or $1$ modulo~$3$, the result can be proved using the fact that a planar graph contains a vertex of degree at most 5.  
We present here the proof in the case when $k = 2\pmod{3}$, the other cases are similar but require only the fact that there is a vertex of bounded degree.  

Suppose $k = 3m+2$ for some integer $m \ge 1$ and that $G$ is an $n$-vertex, $C_4$-free planar graph with no vertex of degree at most 1.  
Let us label the vertices of $C_k$ as $v_1,v_2,\dots,v_{3m+2}$, consecutively. 
Applying the result of~\cite{mdeg7} we find an edge $\{u,v\}$ such that $d(u)+d(v) \le 7$. 
Let the vertices $u$ and $v$ correspond to $v_1$ and $v_2$ in the $k$-cycle. 
Next choose edges to represent the edges of the cycle $\{v_4,v_6\}$, $\{v_8,v_9\}$, $\{v_{11},v_{12}\}$ and so on.  
Since our graph is $C_4$-free, there are is at most one way to choose the vertices $v_3$, $v_6$, $v_9$ and so on.
It follows that we have at most on the order of $n^{m-1}$ copies of the cycle $C_k$ in the graph.  
Thus by iteratively removing an edge whose vertices have degree summing to at most 7, and then deleting all vertices of degree at most 1 which are created,  we find a total of at most order $n^m$ cycles of length $k$ in $G$, completing the proof.\end{proof}

          \section[]{The Number of Trees in Planar Graph with no even cycle of length at most $2\ell$}

 In this section we prove Theorem~\ref{eventree}, namely that
\[
\ex_{\mathcal{P}}(n,T,\{C_4,C_6,\dots,C_{2\ell}\})=\Theta(n^{\beta_{\ell}(T)}).
\]

\begin{proof}[Proof of Theorem \ref{eventree}] 
First we will show that the result is true for paths. We will make use of the following theorem due to Lam and Verstraete.
\begin{theorem}[\cite{lam2005note}]
\label{thanksTao}
       Let $G$ be a graph containing no even cycles of length at most $2\ell$. 
       There exists a constant $D_\ell$ such that for any $v,u$ vertices of $G$ and $k \leq \ell$ a positive integer, the number of paths from $v$ to $u$ of length $k$ in $G$ is at most $D_\ell$.
       \end{theorem}
It is simple to check that  $\beta_{\ell}(P_k) = 1+\floor{\frac{k+\ell-1}{\ell+1}}$.
Now we will prove that \[\ex_{\p}(n,P_k,\{C_4,C_6,\dots,C_{2\ell}\}) = O\left(n^{1+\floor{\frac{k+\ell-1}{\ell+1}}}\right).\]
Let $G$ be an $n$-vertex planar graph containing no even cycle of length at most $2\ell$, for each $k$-vertex path $v_1,v_2,\dots,v_{k+1}$ in $G$. 
Suppose $k\geq 2$ otherwise we are already done, and consider the edges $\{v_{(\ell+1)i+1},v_{(\ell+1)i+2}\}$ for $i=0,1,\dots,\floor{\frac{k+\ell-1}{\ell+1}}-1$ and the edge $\{v_k,v_{k+1}
\}$. 
To bound the number of paths, we notice that, we have a linear number of choices for each of these edges, and in total the number of choices is of order $n^{1+\floor{\frac{k+\ell-1}{\ell+1}}}$. 
After choosing the edges, by Theorem~\ref{thanksTao} there is a constant number of ways to add the paths between two consecutive edges. 
Therefore $\N(G,P_k) = O\left(n^{1+\floor{\frac{k+\ell-1}{\ell+1}}}\right).$

Now let $T$ be any tree. Let us partition the vertex set $V(T)$ into five sets $A_1$, $A_2$, $A_2'$, $A_2''$ and $A_{\geq 3}$. First we partition the set of vertices of degree not equal to 2 as follows.
        \[A_1=\Big\{v\in V(T) \Big| d(v)=1\Big\} \mbox{ and }A_{\geq3}=\Big\{v\in V(T) \Big| d(v)\geq  3 \Big\}.\]
         In particular, $A_1$ is the set of leaves of the tree $T$. Now we will partition the set of vertices of degree equal to 2 into three sets.  Let  
        \[A_2=\Big\{v\in V(T) \Big| d(v)=2 \mbox{ and there is no vertex of $A_\geq{3}$ at distance less than } \ell \mbox{ from } v  \Big\}.\]
        
        Now consider every path $P$ in $T$ such that: 
        \item[$(i)$] Both end vertices of $P$ have degree at least 3 in $T$.
        \item[$(ii)$] The length of $P$ is at least $\ell+1$, but at most $2\ell-1$.
        \item[$(iii)$] Every internal vertex of $P$ has degree 2 in $T$.
        For each such path let $f(P)$ be the middle vertex of $P$, if $P$ has odd length, take either of the two middle vertices.
        
        Let $A_2'$ be the set of consisting of the vertices $f(P)$ for the paths $P$ defined above. Finally, define
        \[A_2''=\Big\{v\in V(T) \Big| d(v)=2, v\notin A_2 \cup A_2'\Big\}.\]

      Let $F$ be the subgraph of $T$ induced by the vertex set $V_1=A_1\cup A_2 \cup A_2'$. Note that $F$ is a path forest and suppose  $F=P_{i_1} \cup P_{i_2} \cup \cdots \cup P_{i_t}$ for paths $P_{i_j}$, $1 \le j \le t$. 
      
      Now we will show that $\beta_\ell(F)=\beta_\ell(T)$.
      Take a set $S$ of vertices and paths which is a witness for the value of $\beta_\ell(T)$ in $T$.  
      Suppose $S$ contains a path $P = x_1,x_2,\dots,x_\ell$ using at least one vertex from $A_2''$. Note that it is not possible for this path to be fully contained in $A_2''$. Indeed, if it was contained in $A_2''$ we would be able to find an $i$ such that $x_i$ is at distance $\ell-1$ from a vertex $v \in A_{\geq 3}$ and $x_{i+1}$ is at distance $\ell-1$ from a vertex $u \in A_{\geq 3}$ and so $P$ would be contained in a  path of length $2\ell+1$ from $v$ to $u$. It follows that the middle vertex of this path must be a vertex of $P$.
      So we may replace the choice of $P$ in $S$ with a vertex (leaf) $f(P)$ in $F$.  
      We obtain that $\beta_\ell(F) \ge \beta_\ell(T)$. 
      
      Next we show $\beta_\ell(F) \le \beta_\ell(T)$.  
      Take a set $S$ of vertices and paths which is a witness for $\beta_\ell(F)$ in $F$. 
      If $P$ is a length $\ell$ path in $S$, by definition, no end vertex is a leaf in $F$, so $P$ is still a length $\ell$ path in $T$ such that every vertex has degree 2. 
      If $v$ is a leaf in $F$, but is no longer a leaf in $T$, then $v$ must have degree 2 in $T$, and there are two possibilities.
      Suppose $v \in A_2''$, then there exists a path $P$ of length at least $\ell+1$ with internal vertices of degree 2 containing $v$. In this case we may replace the choice of $v$ in $S$ by a subpath of $P$ of length $\ell-1$ without using the end vertices of $P$.
      Suppose $v$ is in $A_2$, since one neighbor of $v$ is not in $F$, it must be in $A_{\geq 3} \cup A_2''$, but by definition of $A_2$, then $v$ must have a neighbor in $A_2$ and so $v$ is at distance $\ell$ of $A_{\geq 3}$. Let $u$ be the closest neighbor of $v$ in $u$, and let $P$ be the length $\ell$ path from $v$ to $u$, then we may replace the choice of $v$ in $S$ by $P'$ the path obtained from $P$ by deleting $u$. 
      It follows that $\beta_2(F) \le \beta_2(T)$.
      
      Now let $G$ be an $n$-vertex planar graph containing no even cycle of length at most $2\ell$ and fix a copy of $F$ in $G$. By Lemma~\ref{Constant-Lemma} since every leaf of $T$ is already fixed, we have a bounded number of choices to embed the vertices of $A_{\geq 3}$ in $G$ such that together with the copy of $F$ the embedding can be completed to a copy of $T$. 
      For any given embedding of $F$ and $A_{\geq 3}$, let $x \in A_2''$. By the definition of $A_2''$, there is a vertex $a \in A_{\geq 3}$ with distance less than $\ell$ to $x$. 
      If there are two choices for $a$, pick one which is closest to $x$, on the branch from $x$ that does not contain $a$, and let $b$ be the closest vertex of $V \setminus A_2''$. We show that the distance between $a$ and $b$ is at most $\ell$.
      Suppose by contradiction $b$ is at distance more than $\ell$ from $a$, then pick $c$ to be the vertex between $x$ and $b$ at distance exactly $\ell$ from $a$. 
      It follows that $c \in A_2''$, so there is a vertex $d \in A_{\geq 3}$ in the same branch from $x$ as $b$ and $c$, at distance at most $\ell$ from $c$. 
      Hence the path from $a$ to $d$ has length at most $2\ell-1$ and its middle vertex $y$ is in $F$ and it is closer to $x$ than $y$, but this contradicts the choice of $b$ since $y$ is also in the branch from $x$ not containing $a$.    
      By Theorem~\ref{thanksTao}, there is a constant number of choices to embed $x$ in $G$ such that the embedding can be completeted to $T$, since for each such embedding we have a path of length at most $\ell$ form the corresponding vertices of $a$ and $b$ in $G$.
        \end{proof}       
       
  \bigskip
\noindent \textbf{Acknowledgements.}  We thank Kevin Hendrey for providing a reference to the fourth author for the fact that a $C_4$-free planar graph with no leaf contains an edge whose degree sum is bounded by a constant. We also thank Tony Huynh for providing us with the reference~\cite{eppstein}.  

The research of the first, the third and the fifth authors is partially supported by the National Research, Development and Innovation Office -- NKFIH, grant K 132696. 
The research of the third author is partially supported by  Shota Rustaveli National Science Foundation of Georgia SRNSFG, grant number DI-18-118.   The research of the fourth author is supported by IBS-R029-C1.

\end{document}